\documentclass[review,3p,authoryear,sort]{elsarticle}
\usepackage{amssymb}
\usepackage[tbtags,fleqn]{amsmath}
\usepackage{cases}
\usepackage{amsthm}
\usepackage{enumitem}
\usepackage{fancyhdr}
\usepackage{bbm}
\usepackage{hyperref}
\hypersetup{%
  pdftitle={$L^p~(p>1)$ solutions of BSDEs with generators satisfying some non-uniform conditions in $t$ and $\omega$},%
  pdfsubject={BSDEs, Existence and uniqueness},%
  pdfauthor={Yajun LIU},%
  pdfkeywords={Backward stochastic differential equation, Existence and uniqueness, Comparison theorem, Minimal solution, Non-uniform condition in $(t,\omega)$},%
  pdfstartview=FitH,
  CJKbookmarks=true,%
  bookmarksnumbered=true,%
  bookmarksopen=true,%
  colorlinks=true, linkcolor=blue, urlcolor=blue, citecolor=blue, %
}

 \newcommand{\F}{\mathcal{F}}
 \newcommand{\E}{\mathbf{E}}

 \newcommand{\R}{{\bf R}}

\DeclareMathOperator*{\sgn}{sgn}
\newtheorem{thm}{Theorem}[section]
\newtheorem{lem}{Lemma}[section]
\newtheorem{df}{Definition}[section]
\newtheorem{rem}{Remark}[section]
\newtheorem{cor}{Corollary}[section]
\newtheorem{pro}{Proposition}[section]

\makeatletter
\def\ps@pprintTitle{%
     \let\@oddhead\@empty
     \let\@evenhead\@empty
     \def\@oddfoot{\footnotesize\itshape
       Preprint submitted to \ifx\@journal\@empty
       \else\@journal\fi\hfill\quad \today}%
     \let\@evenfoot\@oddfoot}
\makeatother

\journal{ArXiv}
\begin{document}


\begin{frontmatter}

\title {{\bfseries  $L^p~(p>1)$ solutions of BSDEs with generators satisfying some non-uniform conditions in $t$ and $\omega$}\tnoteref{1}}

\tnotetext[1]{Supported by the National Natural Science Foundation of China (No. 11371362) and the Fundamental Research Funds for the Central Universities (No. 2012QNA36).}

\author{Yajun LIU}
\author{Depeng LI}
\author{Shengjun FAN\corref{cor1}}%
\ead{f\_s\_j@126.com}
\cortext[cor1]{Corresponding author}

\address{College of Sciences, China University of Mining and Technology, Xuzhou, Jiangsu, 221116, PR China}


\begin{abstract}
This paper is devoted to the $L^p$ ($p>1$) solutions of one-dimensional backward stochastic differential equations (BSDEs for short) with general time intervals and generators satisfying some non-uniform conditions in $t$ and $\omega$. An existence and uniqueness result, a comparison theorem and an existence result for the minimal solutions are respectively obtained, which considerably improve some known works. Some classical techniques used to deal with the existence and uniqueness of $L^p$ ($p>1$) solutions of BSDEs with Lipschitz or linear-growth generators are also developed in this paper.
\end{abstract}

\end{frontmatter}


\section {Introduction}
\numberwithin{equation}{section}
Let us fix a extended real number $0\leq T\leq +\infty$, which can be finite or infinite. Let $(\Omega,\F,P)$ be a probability space carrying a standard $d$-dimensional Brownian motion $(B_t)_{t\geq0}$ and $(\F_t)_{t\geq0}$ be the natural $\sigma$-algebra generated by $(B_t)_{t\geq0}$. We assume that $\F_T=\F$ and $(\F_t)_{t\geq 0}$ is right-continuous and complete. In this paper, we are concerned with the following one-dimensional backward stochastic differential equation (BSDE for short in the remaining):
 \begin{flalign}\label{bsde}
    y_t=\xi+\int_t^T g(s,y_s,z_s){\rm d}s-\int_t^T z_s\cdot{\rm d}B_s,~t\in[0,T],
\end{flalign}
where the extended real number~$T$ is called the terminal time, $\xi$ is a one-dimensional $\mathcal{F}_T$-measurable random variable called the terminal condition, the random function $g(\omega,t,y,z):\Omega\times [0,T]\times \R\times \R^d\mapsto\R$ is $(\mathcal{F}_t)$-progressively measurable for each $(y,z)$ called the generator of BSDE (\ref{bsde}). The solution $(y_t,z_t)_{t\in [0,T]}$ is a pair of $(\mathcal{F}_t)$-progressively measurable processes and the triple $(\xi,T,g)$ is called the parameters of BSDE (\ref{bsde}). BSDE with the parameters $(\xi,T,g)$ is usually denoted by BSDE $(\xi,T,g)$.

The nonlinear BSDEs were initially introduced by \citet{Pardoux90)}. They proved an existence and uniqueness result for $L^2$ solutions of multidimensional BSDEs. In their work, the assumptions of generator $g$ is Lipschitz continuous with respect to $(y,z)$ uniformly in $(t,\omega)$, and the terminal time $T$ is finite, the terminal condition $\xi$ and the process $\{g(t,0,0)\}_{t\in [0,T]}$ are square integrable. From then on, BSDEs have been extensively studied and many applications have been found in mathematical finance, stochastic control, partial differential equations and so on (see \citet{El Karoui97} and \citet{Mor09} for details). On the other hand, many papers have been devoted to relaxing the uniform Lipschitz condition on the generator $g$, improving the finite terminal time into the infinite case and studying the solutions under non-square integrable parameters.

Many works including \citet{Mao95}, \citet{Lepeltier97}, \citet{Bahlali01}, \citet{Briand03}, \citet{Hamadene03}, \citet{Briand07}, \citet{Briand08}, \citet{Wang09}, \citet{Chen10}, \citet{Delb11}, \citet{Ma13}, \citet{Hu2015} and \citet{Fan16}, see also the references therein, weakened the uniform Lipschitz condition on the generator $g$, and some of them investigated the $L^p~(p>1)$ solution of BSDE (\ref{bsde}). \citet{Chen00} first improved the result of \citet{Pardoux90)} to the infinite time interval case and proved an existence and uniqueness result for the $L^2$ solution of BSDE (\ref{bsde}) where the generator $g$ is Lipschitz continuous in $(y,z)$ non-uniformly with respect to $t$. Furthermore, \citet{Fan and Jiang10} and \citet{Fan Jiang and Tian11} relaxed the Lipschitz condition of \citet{Chen00} and obtained two existence and uniqueness results for the $L^2$ solution of BSDE (\ref{bsde}) with finite and infinite time intervals, which also generalizes the results of \citet{Mao95} and \citet{Lepeltier97} respectively.

We especially mention that \citet{El Karoui and Huang97} first introduced a stochastic Lipschitz condition of the generator $g$ in $(y,z)$, where the Lipschitz constant depends also on $(t,\omega)$. They investigated a general time interval BSDE driven by a general c\`{a}dl\`{a}g martingale, and some stronger integrability conditions on the generator and terminal condition as well as on the solutions make it possible to replace the uniform Lipschitz condition by a stochastic one. In this spirit, \citet{Bender00} and \citet{Wang07} respectively proved an existence and uniqueness result for the $L^2$ solution and $L^p~(p>1)$ solution of BSDE (\ref{bsde}) with a general time horizon. After that, \citet{Briand08} introduced another stochastic Lipschitz condition involving a bounded mean oscillation martingale and investigated the $L^p$ (for some certain $p>1$) solution of a infinite dimensional BSDE, where some new higher order integrability conditions on the generator and terminal condition (see their assumptions A3 and A4 for details) need to be satisfied.

Motivated by these results, in this paper, we first put forward a new stochastic Lipschitz condition (see (H1) in Section 3) and prove an existence and uniqueness result of the $L^p~(p>1)$ solution of BSDE (\ref{bsde}) with a finite and infinite time interval (see Theorem \ref{existence and uniqueness}). We do not impose any stronger integrability conditions to the parameters $(\xi,g)$ and the solution $(y,z)$ as made in \citet{El Karoui and Huang97}, \citet{Bender00} and \citet{Wang07}, and the integrability condition (\ref{beasic assumption}) is the only requirement in (H1). By introducing an example, we also show that our stochastic Lipschitz condition is strictly weaker than the Lipschitz condition non-uniformly in $t$ used in \citet{Chen00} (see Example 3.1). And by using stopping times to subdivide the interval $[0,T]$, we successfully overcome a new difficulty arisen naturally in our framework, see the proof of Theorem \ref{existence and uniqueness}. Furthermore, in Section 4, by developing a method employed in \citet{Fan Jiang and Tian11} and \citet{Ma13} we establish a general comparison theorem for the $L^p~(p>1)$ solutions of BSDEs when one of generators satisfies a monotonicity condition in $y$ and a uniform continuity condition in $z$, which are both non-uniform in $(t,\omega)$ (see Theorem \ref{general comparison theorem}). Finally, in Section 5, we prove an existence result of the minimal $L^p~(p>1)$ solution for BSDE (\ref{bsde}) when the generator $g$ is continuous and has a linear growth in $(y,z)$ non-uniform in $(t,\omega)$ (see Theorem \ref{minimal solution}), by improving the method used in \citet{Izumi13} to prove in a direct way that the sequence of solutions of the BSDEs approximated by Lipschitz generators is a Cauchy sequence in $S^p\times M^p$. And, based on Theorem \ref{minimal solution} together with Theorem \ref{general comparison theorem}, we will also give a new comparison theorem of the minimal $L^p~(p>1)$ solutions of BSDEs (see Theorem \ref{theorem 5.2}), and a general existence and uniqueness theorem of $L^p~(p>1)$ solutions of BSDEs (see Theorem \ref{theorem 5.3}).

We would like to mention that our results considerably improve some known works including those obtained in \citet{Pardoux90)}, \citet{Chen00}, \citet{Briand07}, \citet{Chen10} and \citet{Fan Jiang and Tian11} etc. And, some classical techniques used to deal with the existence and uniqueness of $L^p$ ($p>1$) solutions of BSDEs with Lipschitz or linear-growth generators are also developed in this paper.

\section{Notations and lemmas}

In this section, we introduce some basic notations and definitions, which will be used in this paper.  First, we use $|\cdot|$ to denote the norm of Euclidean space $\R^d$. For each subset $A\subset\Omega\times[0,T]$, let $\mathbbm{1}_A=1$ in case of $(\omega,t)\in A$, otherwise, let $\mathbbm{1}_A=0$. For each real number $p>1$,~let $L^p(\Omega,\F_T,P;\R)$ be the set of all $\R$-valued and $\F_T$-measurable random variables $\xi$ such that ${\bf E}[|\xi|^p]<+\infty$, and $S^p(0,T;\R)$ (or $S^p$ simply) denote the set of $\R$-valued, adapted and continuous processes $(Y_t)_{t\in[0,T]}$ such that
\begin{equation*}
     \|Y\|_{S^p}:=\left(\E\left[\sup\limits_{t\in[0,T]}|Y_t|^p\right]\right)^\frac{1}{p}<+\infty.
\end{equation*}
In the whole paper, let $M^p(0,T;\R^d)$ (or $M^p$ simply) denote the set of $(\F_t)$-progressively measurable $\R^d$-valued processes $(Z_t)_{t\in[0,T]}$ such that
\begin{equation*}
     \|Z\|_{M^p}:=\left(\E\left[\left(\int_0^T |Z_t|^2{\rm d}t\right)^{\frac{p}{2}}\right]\right)^{\frac{1}{p}}<+\infty.
\end{equation*}
Obviously, both $S^p$ and $M^p$ are Banach spaces for each $p>1$.

Finally, let {\bf{S}} be the set of all nondecreasing continuous functions $\phi(\cdot)$: $\R^+\mapsto\R^+$ with $\phi(0)=0$ and $\phi(x)>0$ for all $x\in \R^+$, here and hereafter $\R^+:=[0,+\infty).$

\begin{df}
A pair of processes $(y_t,z_t)_{t\in [0,T]}$ taking values in $\R\times {\R}^d$ is called a $L^p$ solution of BSDE (\ref{bsde}) for some $p>1$, if $(y_t,z_t)_{t\in [0,T]}\in S^p(0,T;\R)\times M^p(0,T;\R^d)$ and ${\rm d}P-a.s.$, BSDE (\ref{bsde}) holds true for each $t\in[0,T]$.
\end{df}

Let us introduce the following Lemma \ref{lem}, which will be used in Section 3 and Section 5.

\begin{lem}\label{lem}
Let $p>1$, $0\leq T\leq +\infty$, and $(g_t)_{t\in [0,T]}$ is a $(\mathcal{F}_t)$-progressively measurable process such that $\int_0^Tg_t{\rm d}t<+\infty,~{\rm d}P-a.s.$. If $(Y_t,Z_t)_{t\in [0,T]}$ is a $L^p$ solution to the following BSDE:
\begin{flalign}\label{bsde1}
Y_t=Y_T+\int_t^Tg_s{\rm d}s-\int_t^TZ_s\cdot {\rm d}B_s,~~t\in [0,T],
\end{flalign}
then there exists a positive constant $C_p$ depending only on $p$ such that for each $t\in [0,T]$,
\begin{flalign}\label{lem ineq1}
\mathbb{E}\left[\sup\limits_{s\in[t,T]}|Y_s|^p\right]\leq C_p\mathbb{E}\left[|Y_T|^p+\int^T_t\left(|Y_s|^{p-1}|g_s|\right){\rm d}s\right],
\end{flalign}
\begin{flalign}\label{lem ineq2}
\mathbb{E}\left[\left(\int_t^T |Z_s|^2{\rm d}s\right)^{\frac{p}{2}}\right]\leq C_p\left\{\mathbb{E}\left[|Y_T|^p+\left(\int^T_t\left(|Y_s||g_s|\right){\rm d}s\right)^{\frac{p}{2}}\right]+\mathbb{E}\left[\sup\limits_{s\in[t,T]}|Y_s|^p\right]\right\}.
\end{flalign}
Moreover, there exists a positive constant $\bar{C}_p$ depending only on $p$ such that for each $t\in [0,T]$,
\begin{flalign}\label{lem ineq3}
\mathbb{E}\left[\sup\limits_{s\in[t,T]}|Y_s|^p\right]+\mathbb{E}\left[\left(\int_t^T |Z_s|^2{\rm d}s\right)^{\frac{p}{2}}\right]\leq \bar{C}_p\mathbb{E}\left[|Y_T|^p+\left(\int^T_t|g_s|{\rm d}s\right)^p\right]
\end{flalign}
\end{lem}

\begin{proof}
In the same way as Proposition 2.4 in \citet{Izumi13}, we can prove (\ref{lem ineq1}) and (\ref{lem ineq2}). It remains to show (\ref{lem ineq3}). In fact, by basic inequality $2ab\leq a^2+b^2$ and Young's inequality we have, for each constant $\tilde{C}_p>0$,
\begin{flalign}\label{lemproof ineq1}
\tilde{C}_p\mathbb{E}\left[\int^T_t\left(|Y_s|^{p-1}|g_s|\right){\rm d}s\right]&\leq \tilde{C}_p\mathbb{E}\left[\sup\limits_{s\in[t,T]}|Y_s|^{p-1}\cdot\int^T_t|g_s|{\rm d}s\right] \nonumber\\
&\leq \frac{1}{2}\mathbb{E}\left[\sup\limits_{s\in[t,T]}|Y_s|^p\right]+\frac{1}{p}\left(\frac{2(p-1)}{p}\tilde{C}_p\right)^p\mathbb{E}\left[\left(\int_t^T|g_s|{\rm d}s\right)^p\right]
\end{flalign}
and
\begin{flalign}\label{lemproof ineq2}
\mathbb{E}\left[\left(\int^T_t\left(|Y_s||g_s|\right){\rm d}s\right)^{\frac{p}{2}}\right]&\leq \mathbb{E}\left[\sup\limits_{s\in[t,T]}|Y_s|^{\frac{p}{2}}\cdot\left(\int^T_t|g_s|{\rm d}s\right)^{\frac{p}{2}}\right]\nonumber\\
&\leq \frac{1}{2}\mathbb{E}\left[\sup\limits_{s\in[t,T]}|Y_s|^p\right]+\frac{1}{2}\mathbb{E}\left[\left(\int_t^T|g_s|{\rm d}s\right)^p\right].
\end{flalign}
Thus, (\ref{lem ineq3}) follows immediately from (\ref{lem ineq1}), (\ref{lem ineq2}), (\ref{lemproof ineq1}) and (\ref{lemproof ineq2}).
\end{proof}

The following technical Lemma \ref{linear growth sequence bounds} comes from Lemma 4 in \citet{Fan and Jiang11}, which will be used in Section 4. It gives a sequence of upper bounds for functions of linear growth.
\begin{lem}\label{linear growth sequence bounds}
Let $\Psi(\cdot):\R^+\mapsto\R^+$ be a nondecreasing function of linear growth, which means that $\Psi(x)\leq K(x+1)~(K>0)$ holds true for all $x\in \R^+$. Then for each $n\geq 1$,
\begin{flalign*}
\Psi(x)\leq (n+2K)x+\Psi\left(\frac{2K}{n+2K}\right)
\end{flalign*}
holds true for each $x\in \R^+$.
\end{lem}

\section{An existence and uniqueness result}

In this section, we will use a stopping time technique involved in subdividing the time interval $[0,T]$ to prove a general existence and uniqueness result for the $L^p~(p>1)$ solution of BSDE (\ref{bsde}), and introduce an example to show that our stochastic Lipschitz condition is strictly weaker than the Lipschitz condition non-uniformly in $t$ used in \citet{Chen00}. First, let us introduce the following assumptions with the generator $g$, where $0\leq T\leq +\infty$ and $p>1$.
\begin{description}
  \item(H1)~$g$ is Lipschtiz continuous in $(y,z)$ non-uniformly with respect to both $t$ and $\omega$ , i.e., there exist two $(\mathcal{F}_t)$-progressively measurable nonnegative processes $\{u_t(\omega)\}_{t\in [0,T]}$ and $\{v_t(\omega)\}_{t\in [0,T]}$ satisfying
       \begin{flalign}\label{beasic assumption}
        \int_0^T\left[u_t(\omega)+v_t^2(\omega)\right]{\rm d}t\leq M,~~{\rm d}P-a.s.
       \end{flalign}
       for some constant $M>0$ such that ${\rm d}P\times {\rm d}t -a.e.$, for each $y_1,y_2\in \R$, $z_1,z_2\in \R^d$,
      \begin{flalign*}
      \left|g(\omega,t,y_1,z_1)-g(\omega,t,y_2,z_2)\right|\leq u_t(\omega)|y_1-y_2|+v_t(\omega)|z_1-z_2|;
      \end{flalign*}
  \item(H2)~$\mathbb{E}\left[\left(\int_0^T|g(\omega,t,0,0)|{\rm d}t\right)^p\right]<+\infty.$
\end{description}
\vspace{0.2cm}
\begin{rem}
It is worth noting that the above (\ref{beasic assumption}) is equivalent to $\left\|\int_0^Tu_t(\omega)+v_t^2(\omega){\rm d}t\right\|_\infty \leq M $. For the sake of convenience, the $\omega$ in $u_t(\omega)$ and $v_t(\omega)$ is usually omitted without confusion.
\end{rem}

\vspace{0.2cm}
The following Theorem \ref{existence and uniqueness} shows an existence and uniqueness result for $L^p(p>1)$ solutions of BSDEs under assumptions (H1) and (H2), which could be seen as a generalization of the results obtained in \citet{Pardoux90)} and \citet{Chen00}, where the $u_t$ and $v_t$ in (H1) do not depend on $\omega$.

\begin{thm}\label{existence and uniqueness}
Assume that $p>1$, $0\leq T\leq +\infty$ and the generator $g$ satisfies assumptions (H1) and (H2). Then for each $\xi\in L^p(\Omega,\mathcal{F}_T,P;\bf R)$, BSDE $(\xi,T,g)$ admits a unique $L^p$ solution.
\end{thm}

\begin{proof}
Assume that $(y_t,z_t)_{t\in[0,T]}\in S^p(0,T;{\bf {R}})\times M^p(0,T;{\bf {R}}^d)$. It follows from (H1) that $|g(s,y_s,z_s)|\leq |g(s,0,0)|+u_s|y_s|+v_s|z_s|$, then from inequality $(a+b+c)^p\leq 3^p(a^p+b^p+c^p)$, H\"{o}lder's inequality and (H2), we have
\begin{flalign*}
\mathbb{E}\left[\left(\int_0^T|g(s,y_s,z_s)|{\rm d}s\right)^p\right]&\leq 3^p\mathbb{E}\left[\left(\int_0^T|g(s,0,0)|{\rm d}s\right)^p\right]+(3M)^p\mathbb{E}\left[\sup\limits_{s\in[0,T]}|y_s|^p\right]\nonumber \\
&\quad+3^pM^{\frac{p}{2}}\mathbb{E}\left[\left(\int_0^T|z_s|^2{\rm d}s\right)^{\frac{p}{2}}\right]<+\infty.
\end{flalign*}
As a result, the process $\left(\mathbb{E}\left[\left.\xi+\int_0^Tg(s,y_s,z_s){\rm d}s\right|\mathcal{F}_t\right]\right)_{0\leq t\leq T}$ is a $L^p$ martingale. It then follows from the martingale representation theorem that there exists a unique process $Z_\cdot\in M^p(0,T;{\bf {R}}^d)$ such that
\begin{flalign*}
\mathbb{E}\left[\left.\xi+\int_0^Tg(s,y_s,z_s){\rm d}s\right|\mathcal{F}_t\right]=\mathbb{E}\left[\xi+\int_0^Tg(s,y_s,z_s){\rm d}s\right]+\int_0^tZ_s\cdot{\rm d}B_s,~~0\leq t\leq T.
\end{flalign*}
Let $Y_t:=\mathbb{E}\left[\left.\xi+\int_t^Tg(s,y_s,z_s){\rm d}s\right|\mathcal{F}_t\right]$, $0\leq t\leq T$. Obviously, $Y_\cdot\in S^p(0,T;{\bf {R}})$, and it is not difficult to verify that the $(y_t,z_t)_{t\in[0,T]}$ is just the unique $L^p$ solution to the following equation:
\begin{flalign}\label{bsde1}
Y_t=\xi+\int_t^Tg(s,y_s,z_s){\rm d}s-\int_t^TZ_s\cdot{\rm d}B_s,~~t\in [0,T].
\end{flalign}
Thus, we have constructed a mapping from $S^p(0,T;{\bf {R}})\times M^p(0,T;{\bf {R}}^d)$ to itself. Denote this mapping by $I:(y_t,z_t)_{t\in [0,T]}\longrightarrow(Y_t,Z_t)_{t\in [0,T]}$.

Now, suppose that $(y^i_t,z^i_t)_{t\in[0,T]}\in S^p(0,T;{\bf {R}})\times M^p(0,T;{\bf {R}}^d)$, and let $(Y^i_t,Z^i_t)_{t\in[0,T]}$ be the mapping of $(y^i_t,z^i_t)_{t\in[0,T]}$, $i=1,2$, that is, $I(y^i_t,z^i_t)_{t\in[0,T]}=(Y^i_t,Z^i_t)_{t\in[0,T]}$, $i=1,2$. We denote
\begin{flalign*}
&\hat{Y}_t:=Y_t^1-Y_t^2,~~\hat{Z}_t:=Z_t^1-Z_t^2,~~\hat{y}_t:=y_t^1-y_t^2,~~\hat{z}_t:=z_t^1-z_t^2,\\
&\hat{g}_t:=g(t,y_t^1,z_t^1)-g(t,y_t^2,z_t^2),~~t\in [0,T].
\end{flalign*}
Then $(\hat{Y}_t,\hat{Z}_t)_{t\in [0,T]}$ is a $L^p$ solution of the following BSDE:
\begin{flalign*}
\hat{Y}_t=\int_t^T\hat{g}_s{\rm d}s-\int_t^T\hat{Z}_s\cdot{\rm d}B_s,~~t\in [0,T].
\end{flalign*}
Furthermore, (\ref{lem ineq3}) of Lemma \ref{lem} yields that there exists a constant $c_p>0$ depending only on $p$ such that for each $t\in [0,T]$,
\begin{flalign*}
\mathbb{E}\left[\sup\limits_{s\in [t,T]}|\hat{Y}_s|^p+\left(\int_t^T|\hat{Z}_s|^2{\rm d}s\right)^{\frac{p}{2}}\right]\leq c_p\mathbb{E}\left[\left(\int_t^T|\hat{g}_s|{\rm d}s\right)^p\right].
\end{flalign*}
Thus, by virtue of (H1) and H\"{o}lder's inequality we can deduce that for each $t\in [0,T]$,
\begin{flalign}\label{ineq7}
&\mathbb{E}\left[\sup\limits_{s\in [t,T]}|\hat{Y}_s|^p+\left(\int_t^T|\hat{Z}_s|^2{\rm d}s\right)^{\frac{p}{2}}\right] \nonumber \\
&\quad\leq c_p\mathbb{E}\left[\left(\left(\int_t^Tu_s{\rm d}s\right)^p+\left(\int_t^Tv_s^2{\rm d}s\right)^{\frac{p}{2}}\right)\left(\sup\limits_{s\in [t,T]}|\hat{y}_s|^p+\left(\int_t^T|\hat{z}_s|^2{\rm d}s\right)^{\frac{p}{2}}\right)\right].
\end{flalign}

In the sequel, we choose a large sufficiently number $N$ such that
\begin{flalign*}
\frac{M}{N}\leq \frac{1}{(4c_p)^{1/p}}\wedge \frac{1}{(4c_p)^{2/p}},
\end{flalign*}
and subdivide the interval $[0,T]$ into some small stochastic intervals like $[T_{i-1},T_i],i=1,\cdot\cdot\cdot N$, by defining the following $(\mathcal{F}_t)$-stopping times:
\begin{flalign*}
&T_0=0; \\
&T_1=\inf\left\{t\geq 0:\int_0^t\left(u_s+v_s^2\right){\rm d}s\geq \frac{M}{N}\right\}\wedge T; \\
&\qquad\vdots \\
&T_i=\inf\left\{t\geq T_{i-1}:\int_0^t\left(u_s+v_s^2\right){\rm d}s\geq \frac{iM}{N}\right\}\wedge T; \\
&\qquad\vdots \\
&T_N=\inf\left\{t\geq T_{N-1}:\int_0^t\left(u_s+v_s^2\right){\rm d}s\geq \frac{NM}{N}\right\}\wedge T=T.
\end{flalign*}
Thus, for any $[T_{i-1},T_i]\subset [0,T],i=1,\cdot\cdot\cdot N$, it follows that
\begin{flalign}\label{stop time estimate}
\left(\int_{T_{i-1}}^{T_i}u_s{\rm d}s\right)^p+\left(\int_{T_{i-1}}^{T_i}v_s^2{\rm d}s\right)^{\frac{p}{2}}\leq \frac{1}{2c_p}.
\end{flalign}
Now, with the help of inequality (\ref{ineq7}), we have
\begin{flalign*}
\mathbb{E}\left[\sup\limits_{s\in [T_{N-1},T]}|\hat{Y}_s|^p+\left(\int_{T_{N-1}}^T|\hat{Z}_s|^2{\rm d}s\right)^{\frac{p}{2}}\right]\leq \frac{1}{2}\mathbb{E}\left[\sup\limits_{s\in [T_{N-1},T]}|\hat{y}_s|^p+\left(\int_{T_{N-1}}^T|\hat{z}_s|^2{\rm d}s\right)^{\frac{p}{2}}\right],
\end{flalign*}
which means that $I$ is a strict contraction from $S^p(T_{N-1},T;{\bf {R}})\times M^p(T_{N-1},T;{\bf {R}}^d)$ into itself. Then $I$ admits a unique fixed point in this space. It follows that there exists a unique $(y_t,z_t)_{t\in[T_{N-1},T]}\in S^p(T_{N-1},T;{\bf {R}})\times M^p(T_{N-1},T;{\bf {R}}^d)$ satisfying BSDE $(\xi,T,g)$ on $[T_{N-1},T]$. That is to say, BSDE $(\xi,T,g)$ admits a unique $L^p$ solution on $[T_{N-1},T]$.

Finally, note that (\ref{stop time estimate}) holds true for $i=N-1$. By replacing $T_{N-1}$, $T$ and $\xi$ by $T_{N-2}$, $T_{N-1}$ and $y_{T_{N-1}}$, respectively, in the above proof, we can obtain the existence and uniqueness for the $L^p$ solution of BSDE $(\xi,T,g)$ on $[T_{N-2},T_{N-1}]$. Furthermore, repeating the above procedure and making use of (\ref{stop time estimate}), we deduce the existence and uniqueness for the $L^p$ solution of BSDE $(\xi,T,g)$ on $[T_{N-3},T_{N-2}],\cdots,[0,T_1]$. The proof of Theorem \ref{existence and uniqueness} is then completed.
\end{proof}

\begin{rem}
It is easy to see that Theorem \ref{existence and uniqueness} holds still true for multidimensional BSDEs.
\end{rem}

The following example shows that assumption (H1) is strictly weaker than the corresponding assumption in \citet{Chen00}. For readers' convenience, we list the assumption of \citet{Chen00} as the following (H1'):
\begin{description}
  \item(H1')~$g$ is Lipschitz continuous in $(y,z)$, non-uniformly in $t$, i.e., there exist two functions $\bar{u}(t),~\bar{v}(t):[0,T]\mapsto \R^+$ satisfying
       \begin{flalign*}
       \int_0^T\left[\bar{u}(t)+\bar{v}^2(t)\right]{\rm d}t<+\infty
       \end{flalign*}
       such that ${\rm d}P\times {\rm d}t -a.e.$, for each $y_1,y_2\in \R$, $z_1,z_2\in \R^d$,
      \begin{flalign*}
      \left|g(\omega,t,y_1,z_1)-g(\omega,t,y_2,z_2)\right|\leq \bar{u}(t)|y_1-y_2|+\bar{v}(t)|z_1-z_2|.
      \end{flalign*}
\end{description}
{\bf{Example 3.1 }}Let $0\leq T \leq +\infty$, and for each $ t_0\in (0,T)$, define the following two stopping times:
\begin{flalign*}
&\tau_1(\omega)=\inf\left\{t>t_0:|B_{t_0}(\omega)|(t-t_0)\geq M/2\right\}\wedge T,\\
&\tau_2(\omega)=\inf\left\{t>t_0:|B_{t_0}(\omega)|^2(t-t_0)\geq M/2\right\}\wedge T.
\end{flalign*}
Consider the generator $\tilde{g}(\omega, t,y,z):=\tilde{u}_t(\omega)|y|+\tilde{v}_t(\omega)|z|$, where
\begin{flalign*}
\tilde{u}_t(\omega)=|B_{t_0}(\omega)|
\mathbbm{1}_{((t_0,\tau_1(\omega)]]}(\omega,t),~~\tilde{v}_t(\omega)
=|B_{t_0}|\mathbbm{1}_{((t_0,\tau_2(\omega)]]}(\omega,t),~~(t,\omega)\in [0,T]\times \Omega.
\end{flalign*}
It is clear that $\tilde{g}$ satisfies assumptions (H1) and (H2) with $u_t=\tilde{u}_t$ and $v(t)=\tilde{v}_t$. Then, by  Theorem \ref{existence and uniqueness} we know that for each $p>1$ and each $\xi \in L^p(\Omega,\mathcal{F}_T,P;\R)$, BSDE $(\xi,T,\tilde{g})$ admits a unique $L^p$ solution.

We especially mention that this $\tilde{g}$ does not satisfy the above assumption (H1'). In fact, if assumption (H1') holds true for $\tilde{g}$, then there exist two deterministic functions $\bar{u}(t),~\bar{v}(t):[0,T]\mapsto \R^+$ such that
\begin{flalign}\label{example 1}
 \tilde{u}_t(\omega)\leq \bar{u}(t),\ \ \  \tilde{v}_t(\omega)\leq \bar{v}(t),
 ~~{\rm d}P\times {\rm d}t -a.e.
 \end{flalign}
 and
 \begin{flalign}\label{example 2}
 \int_0^T\left[\bar{u}(t)+\bar{v}^2(t)\right]{\rm d}t < +\infty.
\end{flalign}
This yields a contradiction which will be shown below. Note first that for each $t\in (t_0,T)$, we have
\begin{flalign*}
\left\{\omega:\tilde{u}_t(\omega)>\bar{u}(t)\right\}&=\left\{\omega:t\leq \tau_1(\omega)~{\rm{and}}~|B_{t_0}(\omega)|>\bar{u}(t)\right\}\\
&=\left\{\omega:|B_{t_0}(\omega)|\leq \frac{M}{2(t-t_0)}~{\rm{and}}~|B_{t_0}(\omega)|>\bar{u}(t)\right\},
\end{flalign*}
and note that $B_{t_0}(\omega)$ is a normal random variable with $zero$-expected value and $t_0$-variance values. If $\bar{u}(t)<\frac{M}{2(t-t_0)}$ for some $t\in (t_0,T)$, then $P\left(\left\{\omega:\tilde{u}_t(\omega)>\bar{u}(t)\right\}\right)>0.$ Using this fact and (\ref{example 1}) we can conclude that
\begin{flalign*}
\bar{u}_t \geq \frac{M}{2(t-t_0)},~~{\rm d}t-a.e.~{\rm in}~(t_0,T).
\end{flalign*}
Thus,
\begin{flalign*}
\int_0^T\bar{u}(t){\rm d}t\geq \frac{M}{2}\int_{t_0}^T\frac{1}{t-t_0}{\rm d}t=+\infty,
\end{flalign*}
which contradicts with (\ref{example 2}).

Hence, our assumption (H1) is strictly weaker than (H1') used in \citet{Chen00}.

\section{A general comparison theorem}
In this section, by developing a method employed in \citet{Fan Jiang and Tian11} and \citet{Ma13} we will prove a general comparison theorem for the $L^p~(p>1)$ solution of BSDE (\ref{bsde}). Let us first introduce the following assumptions, where $0\leq T\leq+\infty$.
\begin{description}
  \item(H3)~$g$ is monotonic in $y$, non-uniformly with respect to both $t$ and $\omega$ , i.e., there exists a $(\mathcal{F}_t)$-progressively measurable nonnegative process $\{u_t(\omega)\}_{t\in [0,T]}$ satisfying
       \begin{flalign*}
        \int_0^Tu_t(\omega){\rm d}t\leq M,~~{\rm d}P-a.s.
       \end{flalign*}
       for some constant $M>0$ such that ${\rm d}P\times {\rm d}t -a.e.$, for each $y_1,y_2\in \R$, $z_1,z_2\in \R^d$,
      \begin{flalign*}
      \sgn(y_1-y_2)\left(g(\omega,t,y_1,z)-g(\omega,t,y_2,z)\right)\leq u_t(\omega)|y_1-y_2|;
      \end{flalign*}
  \item(H4)~$g$ is uniformly continuous in $z$, non-uniformly with respect to both $t$ and $\omega$ , i.e., there exist a linear-growth function $\phi(\cdot)\in \bf{S}$ and a $(\mathcal{F}_t)$-progressively measurable nonnegative process $\{v_t(\omega)\}_{t\in [0,T]}$ satisfying
       \begin{flalign*}
        \int_0^Tv_t^2(\omega){\rm d}t\leq M,~~{\rm d}P-a.s.
       \end{flalign*}
       such that ${\rm d}P\times {\rm d}t -a.e.$, for each $y_1,y_2\in \R$, $z_1,z_2\in \R^d$,
      \begin{flalign*}
      \left|g(\omega,t,y,z_1)-g(\omega,t,y,z_2)\right|\leq v_t(\omega)\phi(|z_1-z_2|).
      \end{flalign*}
      Here and henceforth, we always assume that $0\leq\phi(x)\leq ax+b$ for all $x\in \R^+$. Furthermore, when $b\neq 0$, we also assume that $\int_0^Tv_t(\omega){\rm d}t\leq M$, ${\rm d}P-a.s.$, where $M$ is defined in (H3).
\end{description}

The following Theorem \ref{general comparison theorem} establishes a general comparison theorem for BSDEs under assumptions (H3) and (H4), which generalizes partly Theorem 2 in \citet{Fan Jiang and Tian11}, where the $u_t(\omega)$ and $v_t(\omega)$ in (H3) and (H4) do not depend on $\omega$ and $p=2$, and Lemma 1 in \citet{Ma13}, where the $u_t(\omega)$ and $v_t(\omega)$ need to be bounded processes and $T<+\infty$.

\begin{thm}\label{general comparison theorem}
Let $p>1$, $0\leq T\leq+\infty$, $\xi,\xi'\in L^p(\Omega,\mathcal{F}_T,P;\R)$, $g$ and $g'$ be two generators of BSDEs, and let $(y_t,z_t)_{t\in[0,T]}$ and $(y'_t,z'_t)_{t\in[0,T]}$ be, respectively, a $L^p$ solution to BSDE $(\xi,T,g)$ and BSDE $(\xi',T,g')$. If ${\rm d}P-a.s.$, $\xi\leq\xi'$, $g$ (resp.~$g'$) satisfies (H3) and (H4) and ${\rm d}P\times {\rm d}t -a.e.$, $g(t,y'_t,z'_t)\leq g'(t,y'_t,z'_t)$ (resp.~$g(t,y_t,z_t)\leq g'(t,y_t,z_t)$), then for each $t\in[0,T]$, we have
\begin{flalign*}
{\rm d}P-a.s.,~~y_t\leq y'_t.
\end{flalign*}
\end{thm}

\begin{proof}
Assume that ${\rm d}P-a.s.$, $\xi\leq \xi'$, $g$ satisfies (H3) and (H4) and ${\rm d}P\times {\rm d}t -a.e.$, $g(t,y'_t,z'_t)\leq g'(t,y'_t,z'_t)$. Setting $\hat{y}_t=y_t-y'_t$, $\hat{z}_t=z_t-z'_t$, $\hat{\xi}=\xi-\xi'$, since $g(s,y'_s,z'_s)-g'(s,y'_s,z'_s)$ is non-positive, we have
\begin{flalign*}
g(s,y_s,z_s)-g'(s,y'_s,z'_s)&=g(s,y_s,z_s)-g(s,y'_s,z'_s)+g(s,y'_s,z'_s)-g'(s,y'_s,z'_s)\\
&\leq g(s,y_s,z_s)-g(s,y'_s,z_s)+g(s,y'_s,z_s)-g(s,y'_s,z'_s)
\end{flalign*}
and we deduce, using assumptions (H3) and (H4), that
\begin{flalign}\label{Thm32}
\mathbbm{1}_{\hat{y}_s>0}[g(s,y_s,z_s)-g'(s,y'_s,z'_s)]\leq u_s\hat{y}_s^++\mathbbm{1}_{\hat{y}_s>0}v_s\phi(|\hat{z}_s|).
\end{flalign}
Thus Tanaka's formula with (\ref{Thm32}) leads to the following inequality, with $A_t:=\int_0^tu_s{\rm d}s$,
\begin{flalign}\label{tanaka formula}
e^{A_t}\hat {y}_t^+&\leq e^{A_T}\hat{\xi}^++\int_t^Te^{A_s}\left\{\mathbbm{1}_{\hat{y}_s>0}[g(s,y_s,z_s)-g'(s,y'_s,z'_s)]-u_s\hat {y}_s^+\right\}{\rm d}s
-\int_t^Te^{A_s}\mathbbm{1}_{\hat{y}_s>0}\hat{z}_s\cdot {\rm d}B_s,
\nonumber\\&\leq \int_t^Te^{A_s}\mathbbm{1}_{\hat{y}_s>0}v_s\phi(|\hat{z}_s|){\rm d}s-\int_t^Te^{A_s}\mathbbm{1}_{\hat{y}_s>0}\hat{z}_s\cdot{\rm d}B_s,~~t\in [0,T].
\end{flalign}
Furthermore, note that Lemma \ref{linear growth sequence bounds} with $\Psi(\cdot)=\phi(\cdot)$ and $K=c:=a+b$ yields that
\begin{flalign}\label{Thm33}
\forall~n\geq 1, ~x\in \R^+, ~\phi(x)\leq (n+2c)x+{\bf 1}_{b\neq 0}\phi\left(\frac{2c}{n+2c}\right).
\end{flalign}
where ${\bf 1}_{b\neq 0}=1$ if $b\neq 0$ and ${\bf 1}_{b\neq 0}=0$ if $b=0$. By (\ref{Thm32})-(\ref{Thm33}), we get that, for each $n\geq 1$ and each $t\in [0,T]$,
\begin{flalign}\label{Thm34}
e^{A_t}\hat{y}_t^+&\leq a_n+\int_t^T\left[e^{A_s}\mathbbm{1}_{\hat{y}_s>0}(n+2c)v_s|\hat{z}_s|\right]{\rm d}s-\int_t^Te^{A_s}\mathbbm{1}_{\hat{y}_s>0}\hat{z}_s\cdot {\rm d}B_s \nonumber\\
&=a_n-\int_t^Te^{A_s}\mathbbm{1}_{\hat{y}_s>0}\hat{z}_s\cdot\left[-\frac{(n+2c)v_s\hat{z}_s}{|\hat{z}_s|}\mathbbm{1}_{|\hat{z}_s|\neq 0}{\rm d}s+{\rm d}B_s\right],
\end{flalign}
where, by (H4),
\begin{flalign}\label{Thm35}
a_n={\bf 1}_{b\neq 0}\phi\left(\frac{2c}{n+2c}\right)\cdot\left\|\int_0^Te^{A_s}v_s{\rm d}s\right\|_{\infty}\leq {\bf 1}_{b\neq 0}\phi\left(\frac{2c}{n+2c}\right)\cdot M\cdot e^M\rightarrow 0~{\rm as}~n\rightarrow \infty.
\end{flalign}

In the sequel, let $P_n$ be the probability on $(\Omega,\mathcal{F})$ which is equivalent to $P$ and defined by
\begin{flalign*}
\frac{{\rm d}P_n}{{\rm d}P}:=\exp\left\{(n+2c)\int_0^T\frac{v_s\hat{z}_s}{|\hat{z}_s|}\mathbbm{1}_{|\hat{z}_s|\neq 0}\cdot {\rm d}B_s-\frac{1}{2}(n+2c)^2\int_0^T\mathbbm{1}_{|\hat{z}_s|\neq 0}v^2_s{\rm d}s\right\}.
\end{flalign*}
It is worth noting that ${\rm d}P_n/{\rm d}P$ has moments of all orders since $\int_0^Tv^2(s){\rm d}s\leq M$, ${\rm d}P-a.s.$. By Girsanov's theorem, under $P_n$ the process
\begin{flalign*}
B_n(t)=B_t-\int_0^t\frac{(n+2c)v_s\hat{z}_s}{|\hat{z}_s|}\mathbbm{1}_{|\hat{z}_s|\neq 0}{\rm d}s,~~t\in [0,T]
\end{flalign*}
is Brownian motion. Moreover, the process $\left(\int_0^te^{A_s}\mathbbm{1}_{\hat{y}_s>0}\hat{z}_s\cdot {\rm d}B_n(s)\right)_{t\in [0,T]}$ is a $(\mathcal{F}_n,P_n)$-martingale. Indeed, let $\mathbb{E}_n[X|\mathcal{F}_t]$ represent the conditional expectation of random variable $X$ with respect to $\mathcal{F}_t$ under $P_n$ and let $\mathbb{E}_n[X]\hat{=}\mathbb{E}_n[X|\mathcal{F}_0]$, then from the Burkholder-Davis-Gundy (BDG) inequality and H\"{o}lder's inequality, we have
\begin{flalign*}
\mathbb{E}_n\left[\sup_{0\leq t\leq T}\left|\int_0^te^{A_s}\mathbbm{1}_{\hat{y}_s>0}\hat{z}_s\cdot {\rm d}B_n(s)\right|\right]&\leq 4e^M\mathbb{E}_n\left[\sqrt{\int_0^T|\hat{z}_s|^2{\rm d}s}\right]\\
&\leq 4e^M\mathbb{E}\left[\left(\frac{{\rm d}P_n}{{\rm d}P}\right)^{\frac{p}{p-1}}\right]^{\frac{p-1}{p}}\mathbb{E}\left[\left(\int_0^T|\hat{z}_s|^2{\rm d}s\right)^{\frac{p}{2}}\right]^{\frac{1}{p}}<+\infty.
\end{flalign*}
Thus, by taking the conditional expectation with respect to $\mathcal{F}_t$ under $P_n$ in (\ref{Thm34}), we obtain that for each $n\geq 1$ and $t\in [0,T]$,
\begin{flalign}\label{Thm36}
e^{A_t}\hat{y}_t^+\leq a_n,\ \ {\rm d}P-a.s.
\end{flalign}
And in view of (\ref{Thm35}), it follows that for each $t\in[0,T]$, ${\rm d}P-a.s.$, $y_t\leq y'_t$.\vspace{0.1cm}

Now, let us assume that ${\rm d}P-a.s.$, $\xi\leq \xi'$, $g'$ satisfies (H3) and (H4) and ${\rm d}P\times {\rm d}t -a.e.$, $g(t,y_t,z_t)\leq g'(t,y_t,z_t)$. Then, since $g(s,y_s,z_s)-g'(s,y_s,z_s)$ is non-positive, we have
\begin{flalign*}
g(s,y_s,z_s)-g'(s,y'_s,z'_s)&=g(s,y_s,z_s)-g'(s,y_s,z_s)+g'(s,y_s,z_s)-g'(s,y'_s,z'_s)\\
&\leq g'(s,y_s,z_s)-g'(s,y'_s,z_s)+g'(s,y'_s,z_s)-g'(s,y'_s,z'_s),
\end{flalign*}
and using (H3) and (H4), we know that inequality (\ref{Thm32}) holds still true. Therefore, the same proof as above yields that for each $t\in[0,T]$, ${\rm d}P-a.s.$, $y_t\leq y'_t$. Theorem \ref{general comparison theorem} is proved.
\end{proof}

From Theorem \ref{general comparison theorem}, the following corollary is immediate.
\begin{cor}\label{corollary}
Let $p>1$, $0\leq T\leq+\infty$, $\xi,\xi'\in L^p(\Omega,\mathcal{F}_T,P;\R)$, one of generaors $g$ and $g'$ satisfy assumptions (H3) and (H4), and $(y_t,z_t)_{t\in[0,T]}$ and $(y'_t,z'_t)_{t\in[0,T]}$ be, respectively, a $L^p$ solution to BSDE $(\xi,T,g)$ and BSDE $(\xi',T,g')$. If ${\rm d}P-a.s.$, $\xi\leq\xi'$, and ${\rm d}P\times {\rm d}t -a.e.$, $g(t,y,z)\leq g'(t,y,z)$ for any $(y,z)\in \R\times\R^d$, then for each $t\in[0,T]$,
${\rm d}P-a.s.,~y_t\leq y'_t.$
\end{cor}

\section{ An existence result of the minimal solutions}

In this section, we will put forward and prove an existence result of the minimal $L^p~(p>1)$ solution for BSDE (\ref{bsde})---Theorem \ref{minimal solution}, by improving the method used in \citet{Izumi13} to prove in a direct way that the sequence of solutions of the BSDEs approximated by the Lipschitz generators is a Cauchy sequence in $S^p\times M^p$. And, based on Theorem \ref{minimal solution} together with Theorem \ref{general comparison theorem}, we will also give a new comparison theorem of the minimal $L^p~(p>1)$ solutions of BSDEs (see Theorem \ref{theorem 5.2}), and a general existence and uniqueness theorem of $L^p~(p>1)$ solutions of BSDEs (see Theorem \ref{theorem 5.3}).
First, we introduce the following assumptions with respect to the generator $g$, where $0\leq T\leq +\infty$.

\begin{description}
  \item(H5)~$g$ has a linear growth in $(y,z)$, non-uniformly with respect to both $t$ and $\omega$ , i.e., there exist three
      $(\mathcal{F}_t)$-progressively measurable nonnegative processes $\{u_t(\omega)\}_{t\in [0,T]}$, $\{v_t(\omega)\}_{t\in [0,T]}$ and $\{f_t(\omega)\}_{t\in[0,T]}$ satisfying
      \begin{flalign*}
      \mathbb{E}\left[\left(\int^T_0f_t(\omega){\rm d}t\right)^p\right]<+\infty,
      \end{flalign*}
      and
      \begin{flalign*}
      \int_0^T\left[u_t(\omega)+v_t^2(\omega)\right]{\rm d}t\leq M,~~{\rm d}P-a.s.,
      \end{flalign*}
      for some constant $M>0$ such that ${\rm d}P\times {\rm d}t -a.e.$, for each $y\in \R$, $z\in \R^d$,
      \begin{flalign*}
      |g(\omega,t,y,z)|\leq f_t(\omega)+u_t(\omega)|y|+v_t(\omega)|z|;
      \end{flalign*}
  \item(H6)~${\rm d}P\times {\rm d}t -a.e.$, $g(\omega,t,\cdot,\cdot):\R\times \R^d\mapsto \R$ is a continuous function.
\end{description}

The following Proposition \ref{pro4i} will play an important role in the proof of Theorem \ref{minimal solution}. Its proof is analogous to Lemma 1 in \citet{Lepeltier97}, so we omit it here.

\begin{pro}\label{pro4i}
Assume that the generator $g$ satisfies assumptions (H5) and (H6). Let $g_n$ be the function defined as follows:
\begin{flalign*}
g_n(\omega,t,y,z):=\inf\limits_{(\bar{y},\bar{z})\in R^{1+d}}\left\{g(\omega,t,\bar{y},\bar{z})+nu_t(\omega)|y-\bar{y}|+nv_t(\omega)|z-\bar{z}|\right\}.
\end{flalign*}
Then the sequence of function $g_n$ is well defined, for each $n\geq 1$, $g_n(\omega,t,y,z)$ is $(\mathcal{F}_t)$-progressively measurable for each $(y,z)\in \R\times \R^d$, and it satisfies, ${\rm d}P\times {\rm d}t -a.e.$,
\begin{description}
  \item(i)~Stochastic linear growth: $\forall~ y,z,~|g_n(\omega,t,y,z)|\leq f_t(\omega)+u_t(\omega)|y|+v_t(\omega)|z|$;
  \item(ii)~Monotonicity in $n$: $\forall~ y,z,~g_n(\omega,t,y,z)$ increases in $n$;
  \item(iii)~Lipschitz condition: $\forall~ y_1,y_2,z_1,z_2$, we have
  \begin{flalign*}
  |g_n(\omega,t,y_1,z_1)-g_n(\omega,t,y_2,z_2)|\leq nu_t(\omega)|y_1-y_2|+nv_t(\omega)|z_1-z_2|;
  \end{flalign*}
  \item(iv)~Convergence: If $(y_n,z_n)\rightarrow(y,z)$, then $g_n(\omega,t,y_n,z_n)\rightarrow g(\omega,t,y,z)$, as $n\rightarrow \infty$.
\end{description}
\end{pro}

Now we state the main result of this section \textemdash Theorem \ref{minimal solution}. It improves Theorem 1 in \citet{Fan Jiang and Tian11}, where the $u_t(\omega)$ and $v_t(\omega)$ in (H5) do not depend on $\omega$, and $p=2$, and Theorem 3.3 in \citet{Izumi13}, where the $u_t(\omega)$ and $v_t(\omega)$ need to be bounded processes and $T<+\infty$.

\begin{thm}\label{minimal solution}
Assume that $p>1$, $0\leq T\leq +\infty$ and that the generator $g$ satisfies (H5) and (H6). Then for each $\xi\in L^p(\Omega, \mathcal{F}_T,P;\R)$, BSDE $(\xi,T,g)$ admits a minimal $L^p$ solution $(y_t,z_t)_{t\in [0,T]}$, which means that if $(\bar{y}_t,\bar{z}_t)_{u\in [0,T]}$ is any $L^p$ solution to BSDE $(\xi,T,g)$, then for each $t\in [0,T]$, ${\rm d}P-a.s.$, $y_t\leq \bar{y}_t$.
\end{thm}

\begin{proof}
Let $g_n$ be defined as in Proposition \ref{pro4i}. In view of (i) of Proposition \ref{pro4i}, for each $n\geq 1$, we have
\begin{flalign*}
\mathbb{E}\left[\left(\int_0^T|g_n(s,0,0)|{\rm d}s\right)^p\right]\leq \mathbb{E}\left[\left(\int_0^Tf_s{\rm d}s\right)^p\right]< +\infty.
\end{flalign*}
In view of (iii) of Proposition \ref{pro4i} and (H5), it follows from Theorem \ref{existence and uniqueness}, that for each $n\geq 1$, BSDE $(\xi,T,g_n)$ and BSDE $(\xi,T,h)$ admit unique $L^p$ solutions $(y^n_t,z^n_t)_{t\in [0,T]}$ and $(Y_t,Z_t)_{t\in [0,T]}$, respectively, where $h(\omega,t,y,z):=f_t(\omega)+u_t(\omega)|y|+v_t(\omega)|z|$ for each $(\omega,t,y,z)$. And in view of (ii) of Proposition \ref{pro4i}, Corollary \ref{corollary} yields that for each $n\geq 1$ and $t\in [0,T]$, $y_t^1(\omega)\leq y_t^n(\omega)\leq y_t^{n+1}(\omega)\leq Y_t(\omega)$, ${\rm d}P-a.s.$. Thus, there must exist a $(\mathcal{F}_t)$-progressively measurable  process $(y_t)_{t\in [0,T]}$ satisfying that for each $t\in [0,T]$,
\begin{flalign*}
\lim_{n\rightarrow +\infty}y_t^n(\omega)=y_t(\omega), ~~{\rm d}P-a.s.,
\end{flalign*}
and for each $n\geq 1$,
\begin{flalign}\label{40}
|y_t^n(\omega)|\leq |y_t^1(\omega)|+|Y_t(\omega)|,~~{\rm d}P-a.s..
\end{flalign}
Now, let $G(\omega)=\sup\limits_{t\in[0,T]}(|y_t^1(\omega)|+|Y_t(\omega)|)$, we have
\begin{flalign}\label{41}
\mathbb{E}\left[\sup\limits_{t\in [0,T]}|y_t|^p\right]\leq \mathbb{E}\left[G^p\right]< +\infty.
\end{flalign}

Furthermore, it follows form (\ref{lem ineq2}) of Lemma \ref{lem} together with (\ref{40}) and (\ref{41}) that there exists a constant $C_p>0$ depending only on $p$ such that for each $n\geq 1$,
\begin{flalign}\label{400}
\mathbb{E}\left[\left(\int_0^T|z_s^n|^2{\rm d}s\right)^{\frac{p}{2}}\right]\leq C_p\mathbb{E}\left[|\xi|^p+\left(\int_0^T\left(|y_s^n||g_n(s,y_s^n,z_s^n)|\right){\rm d}s\right)^{\frac{p}{2}}\right]+C_p\mathbb{E}\left[G^p\right].
\end{flalign}
On the other hand, in view of (i) of Proposition \ref{pro4i} and by inequalities $(a+b+c)^p\leq 3^p(a^p+b^p+c^p)$, $ab\leq \varepsilon a^2+b^2/\varepsilon$ and H\"{o}lder's inequality, we can deduce that for each $n\geq 1$ and $\varepsilon>0$,
\begin{flalign}\label{401}
&\mathbb{E}\left[\left(\int^T_0\left(|y_s^n||g_n(s,y_s^n,z_s^n)|\right){\rm d}s\right)^{\frac{p}{2}}\right] \nonumber\\
&\quad \leq 3^{\frac{p}{2}}\mathbb{E}\left[\left(\int^T_0\left(|y_s^n|f_s\right){\rm d}s\right)^{\frac{p}{2}}+\left(\int^T_0\left(|y_s^n|^2u_s\right){\rm d}s\right)^{\frac{p}{2}}+\left(\int^T_0\left(|y_s^n||z_s^n|v_s\right){\rm d}s\right)^{\frac{p}{2}}\right] \nonumber\\
&\quad \leq 3^{\frac{p}{2}}\left\{\mathbb{E}\left[\sup\limits_{s\in [0,T]}|y_s^n|^p\right]+\frac{1}{2}\mathbb{E}\left[\left(\int_0^Tf_s{\rm d}s\right)^p\right]
+\frac{1}{2}M^p+\left(\frac{1}{\varepsilon}\right)^{\frac{p}{2}}M^{\frac{p}{2}}\mathbb{E}\left[\sup\limits_{s\in [0,T]}|y_s^n|^p\right]
\right\}\nonumber\\
&\qquad +(3\varepsilon)^{\frac{p}{2}}\mathbb{E}\left[\left(\int_0^T|z_s^n|^2{\rm d}s\right)^{\frac{p}{2}}\right],
\end{flalign}
Now choosing $\varepsilon > 0$ such that $C_p(3\varepsilon)^{\frac{p}{2}}=\frac{1}{2}$, from (\ref{400})-(\ref{401}) together with (\ref{40}) and (\ref{41}), we can conclude that
\begin{flalign}\label{42}
\sup\limits_{n\geq 1}\left\|z_\cdot^n\right\|^p_{M^p}=\sup\limits_{n\geq 1}\mathbb{E}\left[\left(\int_0^T|z_s^n|^2{\rm d}s\right)^{\frac{p}{2}}\right]<+\infty.
\end{flalign}

In the sequel, we will show the $(y_t^n)_{t\in [0,T]}$ is a Cauchy sequence in space $S^p(0,T;\R)$. Note that $(y_\cdot^m-y_\cdot^n,z_\cdot^m-z_\cdot^n)$ satisfies the following equation:
\begin{flalign*}
y_t^m-y_t^n=\int_t^T\left[g_m(s,y_s^m,z_s^m)-g_n(s,y_s^n,z_s^n)\right]{\rm d}s-\int_t^T(z_s^m-z_s^n)\cdot {\rm d}B_s,~~t\in [0,T],
\end{flalign*}
for each $m,n\geq 1$. In view of (H5) and (\ref{lem ineq1}) of Lemma \ref{lem}, we obtain that there exists a constant $c_p$ such that
\begin{flalign}\label{cauchy y}
\left\|y_\cdot^m-y_\cdot^n\right\|^p_{S^p}&\leq 2c_p\mathbb{E}\left[\int^T_0\left[|y_s^m-y_s^n|^{p-1}f_s\right]{\rm d}s\right]+c_p\mathbb{E}\left[\int_0^T\left[|y_s^m-y_s^n|^{p-1}u_s(|y_s^m|+|y_s^n|)\right]{\rm d}s\right]\nonumber\\
&\quad +c_p\mathbb{E}\left[\int_0^T\left[|y_s^m-y_s^n|^{p-1}v_s(|z_s^m|+|z_s^n|)\right]{\rm d}s\right].
\end{flalign}
We can prove that the three terms of right-hand side of the previous inequality tend to zero as $m,n\rightarrow \infty$ respectively. Indeed, by (H5), H\"{o}lder's inequality and (\ref{41}), note that
\begin{flalign*}
\mathbb{E}\left[\int_0^T\left(G^{p-1}f_s\right){\rm d}s\right]=\mathbb{E}\left[G^{p-1}\int_0^Tf_s{\rm d}s\right]\leq \left(\mathbb{E}\left[G^p\right]\right)^{\frac{p-1}{p}}\left(\mathbb{E}\left[\left(\int_0^Tf_s{\rm d}s\right)^p\right]\right)^{\frac{1}{p}}< +\infty,
\end{flalign*}
\begin{flalign*}
\mathbb{E}\left[\left(\int_0^T\left(G^{p-1}u_s\right){\rm d}s\right)^{\frac{p}{p-1}}\right]=\mathbb{E}\left[G^p\left(\int_0^Tu_s{\rm d}s\right)^{\frac{p}{p-1}}\right]\leq \mathbb{E}\left[G^p\right]M^{\frac{p}{p-1}}< +\infty,
\end{flalign*}
\begin{flalign*}
\mathbb{E}\left[\left(\int_0^T\left(G^{2p-2}v_s^2\right){\rm d}s\right)^{\frac{p}{2p-2}}\right]=\mathbb{E}\left[G^p\left(\int_0^Tv^2_s{\rm d}s\right)^{\frac{p}{2p-2}}\right]\leq \mathbb{E}\left[G^p\right]M^{\frac{p}{2p-2}}< +\infty.
\end{flalign*}
Since for each $m,n\geq 1$ and $s\in [0,T]$, ${\rm d}P-a.s., |y_s^m(\omega)-y_s^n(\omega)|^{p-1}\leq 2^{p-1}G^{p-1}(\omega)$, and ${\rm d}P \times {\rm d}t-a.e.,\ y_\cdot^n\rightarrow y_\cdot$ as $n\rightarrow +\infty$, by Lebesgue's dominated convergence theorem we deduce that, as $m,n\rightarrow \infty$,
\begin{flalign}\label{43}
&\mathbb{E}\left[\int_0^T\left(|y_s^m-y_s^n|^{p-1}f_s\right){\rm d}s\right]\rightarrow 0,~~~~\mathbb{E}\left[\left(\int_0^T\left(|y_ s^m-y_s^n|^{p-1}u_s\right){\rm d}s\right)^{\frac{p}{p-1}}\right]\rightarrow 0,\nonumber \\
&\mathbb{E}\left[\left(\int_0^T\left(|y_s^m-y_s^n|^{2p-2}v_s^2\right){\rm d}s\right)^{\frac{p}{2p-2}}\right]\rightarrow 0.
\end{flalign}
Thus, in view of (\ref{40}), (\ref{41}), (\ref{42}) and (\ref{43}), it follows from H\"{o}lder's inequality that, as $m,n\rightarrow \infty$,
\begin{flalign}\label{cauchy y1}
&\mathbb{E}\left[\int_0^T\left[|y_s^m-y_s^n|^{p-1}u_s(|y_s^m|+|y_s^n|)\right]{\rm d}s\right]\nonumber\\
&\quad\leq 2\left(\mathbb{E}\left[G^p\right]\right)^{\frac{1}{p}}\left(\mathbb{E}\left[\left(\int_0^T\left(|y_s^m-y_s^n|^{p-1}u_s\right){\rm d}s\right)^{\frac{p}{p-1}}\right]\right)^{\frac{p-1}{p}}\rightarrow 0
\end{flalign}
and
\begin{flalign}\label{cauchy y2}
&\mathbb{E}\left[\int_0^T\left[|y_s^m-y_s^n|^{p-1}v_s(|z_s^m|+|z_s^n|)\right]{\rm d}s\right]\nonumber\\
&\quad\leq\mathbb{E}\left[\left(\int_0^T\left(|y_s^m-y_s^n|^{2p-2}v_s^2\right){\rm d}s\right)^{\frac{1}{2}}\cdot
 \left(\int_0^T(|z_s^m|+|z_s^n|)^2{\rm d}s\right)^{\frac{1}{2}}\right]\nonumber\\
&\quad\leq\mathbb{E}\left[\left(\int_0^T\left(|y_s^m-y_s^n|^{2p-2}v_s^2\right){\rm d}s\right)^{\frac{p}{2p-2}}\right]^{\frac{p-1}{p}}\cdot
 \mathbb{E}\left[\left(\int_0^T(|z_s^m|+|z_s^n|)^2{\rm d}s\right)^{\frac{p}{2}}\right]^{\frac{1}{p}}\rightarrow 0.
\end{flalign}
Hence, combining (\ref{cauchy y})-(\ref{cauchy y2}), we obtain that
\begin{flalign}\label{cauchy y lim}
\lim_{n\rightarrow \infty}\|y_\cdot^n-y_\cdot\|_{S^p}=0.
\end{flalign}

Furthermore, we prove that $(z_t^n)_{t\in [0,T]}$ is a Cauchy sequence in space $M^p(0,T;\R^d)$. In fact, by (\ref{lem ineq2}) of Lemma \ref{lem}, we know the existence of a constant $\bar{C}_p$ depending only on $p$ such that for each $m,n\geq 1$,
\begin{flalign}\label{cauchy z}
\left\|z_\cdot^m-z_\cdot^n\right\|_{M^p}^p&\leq \bar{C}_p\mathbb{E}\left[\left(\int_0^T
\left[|y_s^m-y_s^n||g_m(s,y_s^m,z_s^m)-g_n(s,y_s^n,z_s^n)|\right]{\rm d}s\right)^{\frac{p}{2}}\right]\nonumber\\
&\quad +\bar{C}_p\left\|y_\cdot^m-y_\cdot^n\right\|_{S^p}^p.
\end{flalign}
On the other hand, by (H5), inequality $(a+b+c)^p\leq 3^p(a^p+b^p+c^p)$ and H\"{o}lder's inequality, we deduce that
\begin{flalign}\label{cauchy z1}
&\mathbb{E}\left[\left(\int_0^T\left[|y_s^m-y_s^n||g_m(s,y_s^m,z_s^m)-g_n(s,y_s^n,z_s^n)|\right]{\rm d}s\right)^{\frac{p}{2}}\right]\nonumber\\
&\quad \leq \mathbb{E}\left[\left(\int_0^T|y_s^m-y_s^n|\left(2f_s+u_s(|y_s^m|+|y_s^n|)+v_s(|z_s^m|+|z_s^n|)\right){\rm d}s\right)^{\frac{p}{2}}\right]\nonumber\\
&\quad \leq 3^{\frac{p}{2}}\left\|y_\cdot^m-y_\cdot^n
\right\|_{S^p}^{\frac{p}{2}}\cdot
\left\{2^{\frac{p}{2}}\mathbb{E}
\left[\left(\int_0^Tf_s{\rm d}s\right)^p\right]^{\frac{1}{2}}+
2^{\frac{p}{2}}\mathbb{E}
\left[G^p\right]^{\frac{1}{2}}\cdot M^{\frac{p}{2}}\right\}\nonumber\\
&\quad \quad+3^{\frac{p}{2}}\left\|y_\cdot^m-y_\cdot^n
\right\|_{S^p}^{\frac{p}{2}}\cdot\mathbb{E}
\left[\left(\int_0^T(|z_s^m|+|z_s^n|)^2{\rm d}s\right)^{\frac{p}{2}}\right]\cdot M^{\frac{p}{4}}.
\end{flalign}
Thus, combining (\ref{42}), (\ref{cauchy y lim}), (\ref{cauchy z}) and (\ref{cauchy z1}), we can conclude that there exists a process $z_.\in M^p(0,T;\R^d)$ such that
\begin{flalign}\label{cauchy z lim}
\lim_{n\rightarrow \infty}\|z_\cdot^n-z_\cdot\|_{M^p}=0.
\end{flalign}

Now, we can choose a subsequence of $\{z_\cdot^n\}$, still denote by itself, such that $\|z_\cdot^n-z_\cdot\|_{M^p}\leq \frac{1}{2^n}$ for each $n\geq 1$. Then
\begin{flalign}{\label{z norm bounded}}
\left\|\sup_n|z^n_\cdot|\right\|_{M^p}&\leq \left\|\sup_n|z^n_\cdot-z_\cdot|\right\|_{M^p}+ \left\||z_\cdot|\right\|_{M^p}\leq\left\|
\sum^{+\infty}_{n=1}|z^n_\cdot-z_\cdot|\right\|
+\left\||z_\cdot|\right\|_{M^p}\nonumber\\
&\leq\sum^{+\infty}_{n=1}\left\||z^n_\cdot-z_\cdot|
\right\|_{M^p}+\left\||z_\cdot|\right\|_{M^p}\leq 1+\left\||z_\cdot|\right\|_{M^p}< +\infty.
\end{flalign}
Denote $H_t(\omega):=f_t(\omega)+u_t(\omega)G(\omega)
+v_t(\omega)\sup\limits_{n}|z_t^n(\omega)|$. By (H5), (\ref{40}), (\ref{41}) and (i) of Proposition \ref{pro4i}, we know that for each $n\geq 1$, ${\rm d}P\times {\rm d}t-a.e.$,
\begin{flalign}\label{46}
|g_n(t,y_t^n,z_t^n)-g(t,y,z)|\leq 2H_t.
\end{flalign}
And by H\"{o}lder's inequality together with (\ref{41}) and (\ref{z norm bounded}), we have
\begin{flalign}\label{47}
\mathbb{E}\left[\left(\int_0^T|H_s|{\rm d}s\right)^p\right]&\leq 3^p\mathbb{E}\left[\left(\int_0^Tf_s{\rm d}s\right)^p\right]+3^p\mathbb{E}\left[G^p\right]M^p \nonumber \\
&\quad+3^p\mathbb{E}\left[\left(\int_0^T\sup_{n\geq 1}|z_s^n|^2{\rm d}s\right)^{\frac{p}{2}}\right]M^{\frac{p}{2}}<+\infty.
\end{flalign}
On the other hand, in view of (\ref{cauchy y lim}), (\ref{cauchy z lim}) and (iv) of Proposition \ref{pro4i}, we can assume that, choosing a subsequence if necessary, as $n\rightarrow \infty$,
\begin{flalign}\label{45}
g_n(t,y^n_t,z^n_t)\rightarrow g(t,y_t,z_t), ~~{\rm d}P\times {\rm d}t -a.e..
\end{flalign}
Thus, by (\ref{46})-(\ref{45}), it follows from Lesbesgue's dominated convergence theorem that
\begin{flalign*}
\lim_{n\rightarrow \infty}\mathbb{E}\left[\left(\int_0^T\left|g_n(s,y_s^n,z_s^n)-g(s,y_s,z_s)\right|{\rm d}s\right)^p\right]=0.
\end{flalign*}
Finally, taking limits in BSDE $(\xi,T,g_n)$ yields that $(y_t,z_t)_{t\in[0,T]}$ is a $L^p$ solution of BSDE $(\xi,T,g)$.\vspace{0.1cm}

It remains to prove that $(y_.,z_.)$ is the minimal $L^p$ solution of BSDE $(\xi,T,g)$, let $(\hat{y}_t,\hat{z}_t)_{t\in [0,T]}$ be any solution of BSDE $(\xi,T,g)$. In view of (ii) and (iii) of Proposition \ref{pro4i}, by Corollary \ref{corollary}, we obtain that ${\rm d}P-a.s.,~y_t^n\leq \hat{y}_t$ for each $t\in [0,T]$ and $n\geq 1$, from which and by letting $n\rightarrow \infty$ we get that for each $t\in [0,T]$, ${\rm d}P-a.s.$, $y_t \leq \hat{y}_t$. The proof of Theorem \ref{minimal solution} is then complete.
\end{proof}

\begin{rem}
In the same way as in Theorem \ref{minimal solution}, we can prove the existence of the maximal $L^p~(p>1)$ solution of BSDE (\ref{bsde}) under assumptions (H5) and (H6).
\end{rem}

By Theorem \ref{general comparison theorem} and the proof of Theorem \ref{minimal solution}, we can easily get the following comparison theorem on the minimal (resp. maximal) $L^p$ solutions of BSDEs.

\begin{thm}\label{theorem 5.2}
Assume that $p>1$, $0\leq T\leq +\infty$, $\xi,\xi'\in L^p(\Omega,\F_T, P; \R)$, and both generators $g$ and $g'$ satisfy (H5) and (H6). Let $(y_\cdot,z_\cdot)$ and $(y'_\cdot,z'_\cdot)$ be, respectively, the minimal (resp. maximal) $L^p$ solution of BSDE $(\xi,T,g)$ and BSDE $(\xi',T,g')$ (recall Theorem 5.1 and Remark 5.1). If ${\rm d}P-a.s., \xi\leq \xi'$ and ${\rm d}P\times{\rm d}t-a.e., g(\omega,t,y,z)\leq g'(\omega,t,y,z)$ for each $(y,z)\in \R\times \R^d$, then for each $t\in [0,T]$,
$${\rm d}P-a.s.,\ \ \  y_t\leq y'_t.$$
\end{thm}

By Theorem \ref{minimal solution} and Theorem \ref{general comparison theorem}, the following Theorem \ref{theorem 5.3} follows immediately, which generalizes Theorem \ref{existence and uniqueness} in Section 3.

\begin{thm}\label{theorem 5.3}
Assume that $p>1$, $0\leq T\leq +\infty$, and the generator $g$ satisfies assumption (H2) and the following assumption (H7):
\begin{description}
\item(H7)~$g$ is Lipschitz continuous in $y$ and uniformly continuous in $z$, non-uniformly with respect to both $t$ and $\omega$ , i.e., there exist a linear-growth function $\phi(\cdot)\in \bf{S}$ and two $(\mathcal{F}_t)$-progressively measurable nonnegative processes $\{u_t(\omega)\}_{t\in [0,T]}$ and $\{v_t(\omega)\}_{t\in [0,T]}$ satisfying
       \begin{flalign*}
        \int_0^T\left[u_t(\omega)+v_t^2(\omega)\right]{\rm d}t\leq M,~~{\rm d}P-a.s.
       \end{flalign*}
       for some constant $M>0$ such that ${\rm d}P\times {\rm d}t -a.e.$, for each $y_1,y_2\in \R$, $z_1,z_2\in \R^d$,
    \begin{flalign*}
      \left|g(\omega,t,y_1,z_1)-g(\omega,t,y_2,z_2)\right|\leq u_t(\omega)|y_1-y_2|+v_t(\omega)\phi(|z_1-z_2|).
      \end{flalign*}
\end{description}
Then for each $\xi\in L^p(\Omega,\mathcal{F}_T,P;\R)$, BSDE $(\xi,T,g)$ admits a unique $L^p$ solution.
\end{thm}



\end{document}